\documentclass[12pt,reqno]{amsart}
\usepackage[a4paper,margin=2.5cm,top=2.5cm,bottom=2.5cm,centering,headheight=3ex,headsep=4ex,vcentering]{geometry}

\usepackage{amssymb,amsfonts,amsmath,amsthm}
\usepackage{orcidlink}

\usepackage[cal=cm,scr=euler]{mathalfa}
\usepackage{microtype}
\usepackage{mlmodern}

\numberwithin{equation}{section}

\theoremstyle{plain}
\newtheorem{theorem}{Theorem}
\newtheorem{lemma}{Lemma}

\theoremstyle{definition}

\theoremstyle{remark}
\newtheorem{remark}{Remark}

\newcommand{\NN}{\mathbb{N}}
\newcommand{\ZZ}{\mathbb{Z}}
\newcommand{\poch}{(q;q)_\infty}

\newcommand{\vtwo}{\nu_2}

\title[On the positivity of truncated pentagonal number series]{On the positivity of truncated pentagonal number series and some conjectures of Merca}

\author[M. P. Saikia]{Manjil P. Saikia\,\orcidlink{0000-0002-2997-6731}}
\address[M. P. Saikia]{Mathematical and Physical Sciences division, School of Arts \& Sciences, Ahmedabad University, Navrangpura, Ahmedabad 380009, Gujarat, India}
\email{manjil.saikia@ahduni.edu.in}

\author[A. Sarma]{Abhishek Sarma\,\orcidlink{0009-0005-0075-8000}}
\address[A. Sarma]{Department of Basic Sciences and humanities, Assam Skill University,  Mangaldai 784125, Assam, India}
\email{abhitezu002@gmail.com}


\begin{document}

\begin{abstract}
Let $\nu_2(m)$ denote the $2$-adic valuation of a positive integer $m$ and
set $N_m=m\bigl(1+\nu_2(m)/2\bigr)$. We prove four conjectures of Merca on
the nonnegativity of truncated pentagonal number series weighted by the
infinite products $\prod_{m\ge 1}(1-q^{2N_m})$ and
$\prod_{r\ge 1}\bigl(q^{2^r r};q^{2^{r+1}r}\bigr)_\infty$. Our method is to regard the exponent map $m\mapsto m(\nu_2(m)+2)$
as a dynamical system on the positive integers: the factors of the
associated quotient link into chains along its forward orbits, and the three
orbits seeded at $1$, $4$ and $5$ are pairwise disjoint and telescope to
exactly $1/\bigl((1-q)(1-q^4)(1-q^5)\bigr)$. The exponent triple $(1,4,5)$
is admissible in the sense of earlier work by Liu, which yields the desired factorization
into two series with nonnegative coefficients. This orbit telescoping technique
appears to be a mechanism complementary to the P\'olya--Szeg\H o criterion
that underlies most existing positivity results of this kind in the literature.
\end{abstract}

\maketitle

\section{Introduction}

Euler's pentagonal number theorem,
\begin{equation}\label{eq:epnt}
      \poch:=\prod_{m\ge 1}(1-q^m)=\sum_{r\in\ZZ}(-1)^r q^{r(3r-1)/2},
\end{equation}
is among the oldest and most influential identities in the theory of
partitions. Here and elsewhere we use the Euler product notation
\[
(a;q)_\infty:=\prod_{i\geq 0}(1-aq^i),
\]
for a formal variable $q$. In 2012, Andrews and Merca \cite{AndrewsMerca2012} initiated the
study of its \emph{truncations}, proving that the series obtained by cutting
the pentagonal sum off after finitely many terms, suitably normalized, has
nonnegative coefficients. This discovery opened up what is by now one of the
most active themes in partition theory: establishing the nonnegativity of
truncated theta series and of the closely related \emph{tails} of such series.

We mention some highlights of the last few years. Zhou \cite{Zhou2024} refined
the Andrews--Merca result and settled several conjectures of Merca
\cite{Merca2021a,Merca2021b,Merca2021c} and of Krattenthaler--Merca--Radu
\cite{KrattMercaRadu} on truncated pentagonal number series; Yao
\cite{Yao2024} proved a conjecture of Ballantine and Merca
\cite{BallantineMerca2023} on truncated sums of $6$-regular partitions; and
Chen and Yao \cite{ChenYao2025} confirmed a family of Merca's conjectures
\cite{Merca2021c} on truncated series built from the Rogers--Ramanujan
functions, showing in fact that fewer denominator factors suffice for
nonnegativity than originally conjectured. Recently, the authors \cite{SaikiaSarma} proved a conjecture of Merca \cite{Merca2026} in a
way slightly different from Zhou's treatment \cite{Zhou2024} of the same case.
Most relevant to the present paper, Liu \cite{Liu2024} established a general
positivity theorem for tails of pentagonal-type series divided by products
$(1-q^a)(1-q^b)(1-q^c)$ for an explicit list of admissible triples $(a,b,c)$,
and used it to resolve two conjectures of Merca
\cite[Conjectures 13 and 15]{Merca2025}. A common thread running through most
of these proofs is a classical criterion of P\'olya and Szeg\H o
\cite{PolyaSzego} on partitions into a finite set of parts, or some variant of
it.

In this paper we first prove a conjecture of Merca of a rather different character (Theorem \ref{thm:main} below),
in which the tail of the pentagonal number series is weighted by an infinite
product whose exponents are governed by $2$-adic valuations. To state it, let
$\vtwo(m)=\max\{a\ge 0:2^a\mid m\}$ denote the $2$-adic valuation of a
positive integer $m$, and following Merca set
\[
  N_m=m\Bigl(1+\tfrac{\vtwo(m)}{2}\Bigr),
  \qquad\text{so that}\qquad
  e(m):=2N_m=m\bigl(\vtwo(m)+2\bigr)\in\ZZ_{\ge 1}.
\]
Merca's function $p_\nu(n)$ is defined by the generating product
\begin{equation}\label{eq:pnu-def}
  \sum_{n\ge 0}p_\nu(n)\,q^n
  =\prod_{m\ge 1}\bigl(1+q^m+q^{2m}+\cdots+q^{\vtwo(2m)\,m}\bigr),
\end{equation}
so that $p_\nu(n)$ counts the partitions of $n$ in which each part $m$ appears
at most $\vtwo(2m)$ times. Since $\vtwo(2m)=\vtwo(m)+1$, the top exponent of
the $m$-th factor in \eqref{eq:pnu-def} is $(\vtwo(m)+1)m$, and summing the
finite geometric series gives
\begin{equation}\label{eq:factor}
  1+q^m+\cdots+q^{(\vtwo(m)+1)m}
  =\frac{1-q^{(\vtwo(m)+2)m}}{1-q^m}
  =\frac{1-q^{e(m)}}{1-q^m}.
\end{equation}
Substituting \eqref{eq:factor} into \eqref{eq:pnu-def}, we obtain
\begin{equation}\label{eq:pnu-closed}
  \sum_{n\ge 0}p_\nu(n)\,q^n
  =\prod_{m\ge 1}\frac{1-q^{e(m)}}{1-q^m}
  =\frac{1}{\poch}\prod_{m\ge 1}\bigl(1-q^{e(m)}\bigr)
  =\frac{1}{\poch}\prod_{m\ge 1}\bigl(1-q^{2N_m}\bigr).
\end{equation}
For an integer $k\ge 1$ write
\[
  g_k:=\frac{k(3k+1)}{2}
\]
for the $k$-th generalized pentagonal number. Our first main result is the following, which was
conjectured by Merca \cite{Merca2024}.

\begin{theorem}[Conjecture 14, \cite{Merca2024}]\label{thm:main}
For every integer $k\ge 1$, set
\[
  C_k(q)
  =(-1)^k
  \Biggl(
    1-\frac{1}{\poch}\sum_{r=1-k}^{k}(-1)^r q^{r(3r-1)/2}
  \Biggr)
  \prod_{m\ge 1}\bigl(1-q^{2N_m}\bigr).
\]
Then $C_k(q)$ has nonnegative coefficients for all $k\geq 1$.

Furthermore, writing $C_k(q)=\sum_{n\ge 0}c_k(n)q^n$, we also have
\(
  c_k(n)>0\) if and only if
\(  n\ge g_k.
\)
\end{theorem}

\noindent The combinatorial interpretation of Theorem \ref{thm:main} was given by \cite[Theorem 15]{Merca2024}.

Using Theorem \ref{thm:main2} we are able to prove another conjecture of Merca. 

\begin{theorem}[Conjecture 16, \cite{Merca2024}]\label{thm:main2}
For every integer $k\ge 1$, set
\begin{align} 
\overline{C}_k(q)=(-1)^{k-1} \left( 1-\frac{1}{(q;q)_\infty } \sum _{n=-k}^k (-1)^n\,q^{n(3n-1)/2}\right) \prod _{n=1}^\infty (1-q^{2N_n}).
\end{align}
Then $\overline{C}_k(q)$ has nonnegative coefficients for all $k\geq 1$.

Furthermore, writing $\overline{C}_k(q)=\sum_{n\ge 0}\overline{c}_k(n)q^n$, we also have
\(
  \overline{c}_k(n)>0\) if and only if
\(  n\ge g_k+(2k+1).
\)
\end{theorem}

\begin{remark}
    In the original conjecture, $\sum _{n=-k}^k (-1)^n\,q^{n(3n-1)/2}$ was incorrectly typed as $\sum _{n=-k}^k (-1)^k\,q^{n(3n-1)/2}$. This was confirmed after a private communication with Merca.
\end{remark}

\noindent The combinatorial interpretation of Theorem \ref{thm:main2} was given by Merca \cite[Theorem 17]{Merca2024}.

\begin{remark}
Liu \cite{Liu2024} proved two of Merca's conjectures
\cite[Conjectures 13 and 15]{Merca2025}, which have the same shape as
Theorem~\ref{thm:main} but with the product $\prod_{n\ge 1}(1-q^{N_{2n}})$ in
place of $\prod_{m\ge 1}(1-q^{2N_m})$. The two products are 
different: the exponents $N_{2n}=n(\vtwo(n)+3)$ begin $3,8,9,20,15,24,\dots$,
whereas the exponents $2N_m=m(\vtwo(m)+2)$ begin $2,6,6,16,10,18,\dots$.
Moreover, as noted in the introduction, our techniques differ from Liu's as
well.
\end{remark}

Let us explain what separates Theorems~\ref{thm:main} and \ref{thm:main2} from the results
surveyed above, and why the existing machinery does not apply to it directly.
The obstacle is the infinite product $\prod_{m\ge 1}(1-q^{2N_m})$: its
exponents $2N_m=m(\vtwo(m)+2)$ depend on the $2$-adic valuation of $m$, so the
product is neither a theta function nor an eta quotient, and the truncation
techniques developed for such series do not apply directly. Our key idea is
to treat the exponent map $e(m)=m(\vtwo(m)+2)$ as a \emph{dynamical system} on
the positive integers. In the quotient \eqref{eq:pnu-closed}, the numerator of
the factor indexed by $m$ coincides with the denominator of the factor indexed
by $e(m)$, so the factors link into chains along the forward orbits of $e$,
and each chain telescopes. The key lemma of the proof of Theorem \ref{thm:main} (Lemma~\ref{lem:H}) is the
observation that the three orbits seeded at $1$, $4$ and $5$ are pairwise
disjoint and collapse to
exactly the product $\frac{1}{(1-q)(1-q^4)(1-q^5)}$, whose exponent triple
$(1,4,5)$ is one of Liu's admissible triples. After a further rewriting of
$C_k(q)$ via Euler's pentagonal number theorem, we are able to give the complete proof of Theorem \ref{thm:main}. The exact positivity bounds need some more work.

We remark that the telescoping argument we use is not specific to the exponent
$\vtwo(m)+2$: it applies to any map $m\mapsto m\,\omega(m)$ for which suitable
seed orbits are pairwise disjoint and collapse onto an admissible triple. In
particular, the family $e_c(m)=m(\vtwo(m)+c)$ contains both the product of the
present paper ($c=2$) and the product appearing in the conjectures proved by
Liu ($c=3$). We expect the cases $c\ge 4$ to yield new positivity results of
the same flavor, and we hope to return to this elsewhere. Our proof can thus
be seen as a new mechanism, complementary to the P\'olya--Szeg\H o criterion used by other authors (see \cite{Yao2024}, \cite{SaikiaSarma}, etc.),
for obtaining positivity results on truncated series.

To further illustrate the applicability of our techniques, we also prove two other conjectures due to Merca \cite{Merca2025b}.

\begin{theorem}[Conjecture 13, \cite{Merca2025b}]\label{thm:main3}
For every integer $k\ge 1$, set
\[
  \tilde C_k(q)
  =(-1)^k
  \Biggl(
    1-\frac{1}{\poch}\sum_{r=1-k}^{k}(-1)^r q^{r(3r-1)/2}
  \Biggr)
  \prod_{r\ge 1}\bigl(q^{2^rr};q^{2^{r+1}r}\bigr)_\infty.
\]
Then $\tilde C_k(q)$ has nonnegative coefficients for all $k\geq 1$.

Furthermore, writing $\tilde C_k(q)=\sum_{n\ge 0}\tilde c_k(n)q^n$, we also have
\(
 \tilde  c_k(n)>0\) if and only if
\(  n\ge g_k.
\)
\end{theorem}
\noindent The combinatorial interpretation of Theorem \ref{thm:main3} was given by Merca \cite[Theorem 14]{Merca2025b}.

Similar to Theorem \ref{thm:main2}, we are able to prove another conjecture of Merca from Theorem \ref{thm:main3}.

\begin{theorem}[Conjecture 15, \cite{Merca2025b}]\label{thm:main4}
For every integer $k\ge 1$, set
\begin{align} 
\hat {C}_k(q)=(-1)^{k-1} \left( 1-\frac{1}{(q;q)_\infty } \sum _{n=-k}^k (-1)^n\,q^{n(3n-1)/2}\right) \prod _{n=1}^\infty \bigl(q^{2^nn};q^{2^{n+1}n}\bigr)_\infty.
\end{align}
Then $\hat{C}_k(q)$ has nonnegative coefficients for all $k\geq 1$.

Furthermore, writing $\hat{C}_k(q)=\sum_{n\ge 0}\hat c_k(n)q^n$, we also have
\(
  \hat{c}_k(n)>0\) if and only if
\(  n\ge g_k+(2k+1).
\)
\end{theorem}
\begin{remark}
    Similar to Theorem \ref{thm:main2}, we have corrected a typo in the original conjecture here.
\end{remark}
\noindent The combinatorial interpretation of Theorem \ref{thm:main4} was given by Merca \cite[Theorem 16]{Merca2025b}.

The paper is organized as follows: Section \ref{sec:reduction} is devoted to the proof of Theorem \ref{thm:main} and illustrates our technique in full details, we then prove Theorems \ref{thm:main3}in Section \ref{sec:main3}, and finally we prove Theorems \ref{thm:main2} and \ref{thm:main4} in Section \ref{sec:new}.

\section{Proof of Theorem~\ref{thm:main}}\label{sec:reduction}

We now set up the two series into which $C_k(q)$ will be factored. Isolate the
polynomial
\[
  D(q):=(1-q)(1-q^4)(1-q^5),
\]
and introduce
\[
  B_k(q):=\frac{1}{D(q)}\sum_{j\notin[-k,k-1]}(-1)^{j+k}q^{j(3j+1)/2},
  \qquad
  H(q):=D(q)\,\frac{\prod_{m\ge 1}(1-q^{2N_m})}{\poch}.
\]
Here and below we use the abbreviation
\[
  \sum_{j\notin[a,b]}X_j:=\sum_{j\le a-1}X_j+\sum_{j\ge b+1}X_j
  \qquad(a<b).
\]
We also take $\mathbb{N}$ to be set of natural numbers (including $0$). In Section~\ref{sec:reduction} we show, using Euler's pentagonal number
theorem, that
\[
  C_k(q)=H(q)\,B_k(q),
\]
and nonnegativity of each factor is established separately:
$B_k\in\NN[[q]]$ follows from a result of Liu \cite[Corollary 2.3]{Liu2024}
(Lemma~\ref{lem:liu}), while $H\in\NN[[q]]$ is proved directly via the orbit
telescoping method described above (Lemma~\ref{lem:H}). The exact positivity
bound for $c_k$ then follows from elementary information about the
low-degree coefficients of $H$ and $B_k$.

\subsection{Some Preliminaries}
We recall the following lemma.
\begin{lemma}\label{lem:liu}
For every integer $k\ge 1$,
\[
  B_k(q)
  =\frac{1}{(1-q)(1-q^4)(1-q^5)}
   \sum_{j\notin[-k,k-1]}(-1)^{j+k}q^{j(3j+1)/2}
  \in\NN[[q]].
\]
\end{lemma}

\begin{proof}
Liu \cite{Liu2024} defines, for distinct pairwise coprime positive integers
$a,b,c$ and an integer-valued polynomial $Aj^2+Bj$ with $A>B\ge 0$, the
coefficients $\gamma^k_{a,b,c,A,B}(n)$ by
\[
  \frac{1}{(1-q^a)(1-q^b)(1-q^c)}
  \sum_{j\notin[-k,k-1]}(-1)^{j+k}q^{Aj^2+Bj}
  =\sum_{n\ge 0}\gamma^k_{a,b,c,A,B}(n)\,q^n,
\]
and \cite[Corollary 2.3]{Liu2024} asserts that $\gamma^k_{a,b,c,3/2,1/2}(n)\ge 0$ for all
$k\ge 1$ and $n\ge 0$ whenever
\[
  (a,b,c)\in\{(1,2,3),(1,2,5),(1,2,7),(1,3,4),(1,3,5),(1,3,8),(1,4,5),(1,4,7)\}.
\]
Taking $(a,b,c)=(1,4,5)$ and $(A,B)=(\tfrac32,\tfrac12)$ gives us $B_k(q)$, so $B_k(q)\in\NN[[q]]$.
\end{proof}

We now turn to the factor $H(q)$, which carries all the arithmetic information
of the product $\prod_{m\ge 1}(1-q^{2N_m})$. Combining the definition of $H$
with \eqref{eq:pnu-closed}, we may write
\begin{equation}\label{eq:H-quotient}
  H(q)=D(q)\sum_{n\geq 0}p_\nu(n)q^n
      =D(q)\prod_{m\ge 1}\frac{1-q^{e(m)}}{1-q^m}.
\end{equation}
At first sight it is not clear why this series should have nonnegative
coefficients: the numerators $1-q^{e(m)}$ contribute negative terms, and the
prefactor $D(q)=(1-q)(1-q^4)(1-q^5)$ only makes matters look worse. The key
observation is that the exponent map $e$ links the factors of
\eqref{eq:H-quotient} into chains. Indeed, the numerator of the factor indexed
by $m$ is $1-q^{e(m)}$, which is precisely the \emph{denominator} of the factor
indexed by $e(m)$. Following these links, long runs of factors cancel, and
it turns out that the three chains started at $m=1$, $m=4$ and $m=5$ collapse
to exactly $\frac1{D(q)}$. This is the reason for the choice of $D$: after the
collapse, the prefactor $D(q)$ in \eqref{eq:H-quotient} is cancelled, and what
survives is a product of polynomials with coefficients $0$ and $1$.

To make this precise we introduce some notation. For $s\ge 1$ define the
iterates of $e$ by $e^{(0)}(s)=s$ and $e^{(r+1)}(s)=e\bigl(e^{(r)}(s)\bigr)$
for $r\ge 0$, and let
\[
  \mathcal{O}_s=\{e^{(r)}(s):r\ge 0\}
\]
denote the forward orbit of $s$ under $e$. Since
$e(m)=m(\vtwo(m)+2)\ge 2m$ for every $m\ge 1$, the iterates strictly
increase. In particular, each orbit $\mathcal{O}_s$ is an infinite strictly
increasing sequence with $e^{(r)}(s)\to\infty$ as $r\to\infty$. For the three
cases relevant to us, direct computation gives
\[
  \mathcal{O}_1=\{1,2,6,18,54,\dots\},\quad
  \mathcal{O}_4=\{4,16,96,672,\dots\},\quad
  \mathcal{O}_5=\{5,10,30,90,\dots\}.
\]
The general structure of these orbits is established in the lemma below.

\begin{lemma}\label{lem:H}
With $H(q)=\sum_{n\ge 0}h(n)q^n$ as in \eqref{eq:H-quotient}, set
$\mathcal{R}=\mathcal{O}_1\cup\mathcal{O}_4\cup\mathcal{O}_5$. Then
\begin{equation}\label{eq:H-product}
  H(q)=\prod_{m\notin\mathcal{R}}
       \bigl(1+q^m+q^{2m}+\cdots+q^{(\vtwo(m)+1)m}\bigr)\in\NN[[q]].
\end{equation}

Moreover $h(0)=h(3)=1$, while $h(1)=h(2)=h(4)=h(5)=h(6)=0$, and
\(
  h(n)>0\) for every $n\ge 7$.
\end{lemma}

\begin{proof}
The proof proceeds in four steps. We first show that the three orbits
$\mathcal{O}_1,\mathcal{O}_4,\mathcal{O}_5$ are pairwise disjoint, so
that regrouping the factors of \eqref{eq:H-quotient} along these orbits is
allowed. We then show that the factors indexed by a single orbit
$\mathcal{O}_s$ telescope down to $\tfrac1{1-q^s}$, then we assemble these pieces
into the product formula \eqref{eq:H-product}, and finally read off
the support of $H$ from \eqref{eq:H-product} by interpreting $h(n)$ as a
restricted partition counting function. 

We will now show that the three orbits are pairwise disjoint. We first describe each orbit explicitly. Since $e(m)$ is a multiple of $m$, so $m$ divides every member of $\mathcal{O}_m$. Hence, it is clear that
\begin{equation}\label{eq:O4}
  \mathcal{O}_4\subseteq 4\ZZ.
\end{equation}
Next consider integers of $2$-adic valuation exactly $1$. If $\vtwo(m)=1$,
then $e(m)=m(\vtwo(m)+2)=3m$, and multiplying by the odd number $3$ does not
change the valuation, so $\vtwo(e(m))=1$ as well. Thus $e$ maps the set
$\{m:\vtwo(m)=1\}$ into itself with a multiplication factor of $3$.
Now $e(1)=1\cdot(0+2)=2$ and $e(5)=5\cdot(0+2)=10$, and $\nu_2(2)=\nu_2(10)=1$, hence, once an orbit enters this set, all later iterates
are obtained by repeated multiplication by $3$. This lets us conclude
\begin{align}
  \mathcal{O}_1&=\{1\}\cup\{2\cdot 3^{r}:r\ge 0\}=\{1,2,6,18,54,\dots\},\label{eq:O1O5}\\
  \mathcal{O}_5&=\{5\}\cup\{10\cdot 3^{r}:r\ge 0\}=\{5,10,30,90,\dots\}.\label{eq:O5}
\end{align}

Every element of
$\mathcal{O}_1\cup\mathcal{O}_5$ is either odd (and then equal to $1$ or $5$)
or of valuation exactly $1$, hence congruent to $2\pmod 4$. In either case it
is not divisible by $4$, so by \eqref{eq:O4}  $\mathcal{O}_4$ does not intersect with either $\mathcal{O}_1$ or $\mathcal{O}_5$. Finally,
$\mathcal{O}_1\cap\mathcal{O}_5=\varnothing$: the odd elements $1$ and $5$ are
distinct, and an equality $2\cdot 3^{r}=10\cdot 3^{s}$ among even elements
would force $3^{r-s}=5$, which is impossible. 

Thus, we have proved that
$\mathcal{O}_1,\mathcal{O}_4 \text{ and }\mathcal{O}_5$ are pairwise disjoint.

We now fix $s\in\{1,4,5\}$ and consider the factors of \eqref{eq:H-quotient} indexed
by the elements of $\mathcal{O}_s$. The factor indexed by $m=e^{(r)}(s)$ is
\[
  \frac{1-q^{e(e^{(r)}(s))}}{1-q^{e^{(r)}(s)}}
  =\frac{1-q^{e^{(r+1)}(s)}}{1-q^{e^{(r)}(s)}},
\]
so the numerator of the $r$-th factor is exactly the denominator of the
$(r+1)$-st. Consequently, for each $R\ge 0$ the partial products cancel out to give us
\[
  \prod_{r=0}^{R}\frac{1-q^{e^{(r+1)}(s)}}{1-q^{e^{(r)}(s)}}
  =\frac{1-q^{e^{(R+1)}(s)}}{1-q^{e^{(0)}(s)}}
  =\frac{1-q^{e^{(R+1)}(s)}}{1-q^{s}}.
\]
Since $e^{(R+1)}(s)\to\infty$ as $R\to\infty$, the numerator
$1-q^{e^{(R+1)}(s)}$ tends to $1$ in the ring of formal power series. Passing to the limit, the product of all factors
indexed by $\mathcal{O}_s$ is therefore
\begin{equation}\label{eq:orbit-collapse}
  \prod_{m\in\mathcal{O}_s}\frac{1-q^{e(m)}}{1-q^m}
  =\frac{1}{1-q^{s}}.
\end{equation}

The index set $\{1,2,3,\dots\}$ of the factors of \eqref{eq:H-quotient} is the disjoint union of
$\mathcal{O}_1$, $\mathcal{O}_4$, $\mathcal{O}_5$ and the complementary set
$\{m\ge 1:m\notin\mathcal{R}\}$. Regrouping the product in
\eqref{eq:H-quotient} and applying
\eqref{eq:orbit-collapse} to each of the three orbits gives us
\[
  \prod_{m\ge 1}\frac{1-q^{e(m)}}{1-q^m}
  =\Biggl(\prod_{s\in\{1,4,5\}}\frac{1}{1-q^{s}}\Biggr)
   \prod_{m\notin\mathcal{R}}\frac{1-q^{e(m)}}{1-q^m}
  =\frac{1}{D(q)}\prod_{m\notin\mathcal{R}}\frac{1-q^{e(m)}}{1-q^m}.
\]
Multiplying both sides by $D(q)$ cancels the prefactor, and rewriting each
remaining factor via \eqref{eq:factor} in its geometric-series form
$1+q^m+\cdots+q^{(\vtwo(m)+1)m}$ yields \eqref{eq:H-product}. Every factor in
\eqref{eq:H-product} is a polynomial with coefficients in $\{0,1\}$, so the
coefficients of $H$ are nonnegative. This proves the first part of the result.

We now want to prove the second part of the result.
Expanding the product \eqref{eq:H-product} shows that $h(n)$ counts the
representations
\[
  n=\sum_{m\notin\mathcal{R}}a_m\,m,
  \qquad 0\le a_m\le \vtwo(m)+1.
\]
In the language of partitions, $h(n)$ is the number of partitions of $n$ into
\emph{allowed parts} $m\notin\mathcal{R}$, where the part $m$ may be repeated
at most $\vtwo(m)+1$ times. In particular, $h(n)>0$ if and only if at least
one such partition of $n$ exists. Note that every odd allowed part may be used
at most once.

We first determine the small values. From \eqref{eq:O1O5} and \eqref{eq:O4} we
read off $\mathcal{R}\cap[1,6]=\{1,2,4,5,6\}$, so the only allowed part not
exceeding $6$ is $3$, and being odd it may be used at most
$\vtwo(3)+1=1$ time. Thus, we have
\begin{itemize}
  \item $h(0)=1$ (the empty partition) and $h(3)=1$ (the single part $3$),
  \item $h(1)=h(2)=0$, since the smallest allowed part is $3$,
  \item $h(4)=h(5)=0$, since the only partitions available use the part $3$
        at most once, and
  \item $h(6)=0$, since $6\in\mathcal{R}$ is not itself allowed, and the only
        other candidate $6=3+3$ repeats the part $3$.
\end{itemize}

We now want to show that $h(n)>0$ for all $n\ge 7$. If $n\notin\mathcal{R}$, then $n$ is
itself an allowed part, and the one-part partition $n=n$ shows $h(n)\ge 1$ for all such $n$.
Now, let $n\in\mathcal{R}$. The elements of $\mathcal{R}$ in the range
$[7,18]$ are $10$, $16$ and $18$, and each admits a partition into two
distinct allowed parts, as shown below
\[
  10=3+7,\qquad 16=7+9,\qquad 18=3+15.
\]
Finally, let
$n\in\mathcal{R}$ with $n\ge 19$. We know that the only odd elements of
$\mathcal{R}$ are $1$ and $5$, so $n$ is even. Then $n-3$ is odd with
$n-3\ge 16$, hence $n-3\notin\{1,5\}$, which forces $n-3\notin\mathcal{R}$.
Thus $n=3+(n-3)$ is a partition of $n$ into two distinct allowed parts, and we have
$h(n)\ge 1$ for all $n\geq 7$. This completes our proof.
\end{proof}

\begin{remark}
The map $e$ is \emph{not} injective. For instance $e(2)=e(3)=6$, so
forward orbits of distinct initial values may merge (the orbit
$\mathcal{O}_3=\{3,6,18,54,\dots\}$ meets $\mathcal{O}_1$ from its second
element onwards). 

This causes no difficulty in the argument we presented above, for two
reasons. Firstly, the cancellations \eqref{eq:orbit-collapse} takes place entirely
\emph{within} a single orbit, where we do not use the notion of injectivity of $e$ on all of $\ZZ_{\ge 1}$. Secondly, our proof requires that the three particular orbits
$\mathcal{O}_1,\mathcal{O}_4,\mathcal{O}_5$ be pairwise disjoint, which was
already proved in the beginning of the argument.
\end{remark}

\subsection{Proof of Theorem~\ref{thm:main}}

We are now ready to prove Theorem \ref{thm:main}. We do this in two parts: first, we show that $c_k(n)\geq 0$ for all $n\geq 0$, and then we look at when we have strict inequality $c_k>0$.

Euler's pentagonal number theorem states
\[
  \poch=\sum_{r\in\ZZ}(-1)^r q^{r(3r-1)/2}.
\]
Subtracting the partial sum over $r\in[1-k,k]$ from this identity gives
\[
  \poch-\sum_{r=1-k}^{k}(-1)^r q^{r(3r-1)/2}
  =\sum_{r\notin[1-k,k]}(-1)^r q^{r(3r-1)/2},
\]
so that
\[
  1-\frac{1}{\poch}\sum_{r=1-k}^{k}(-1)^r q^{r(3r-1)/2}
  =\frac{1}{\poch}\sum_{r\notin[1-k,k]}(-1)^r q^{r(3r-1)/2}.
\]
Multiplying by $(-1)^k\prod_{m}(1-q^{2N_m})$ and absorbing the sign into the
summand yields
\[
  C_k(q)
  =\frac{\prod_{m}(1-q^{2N_m})}{\poch}
   \sum_{r\notin[1-k,k]}(-1)^{r+k}q^{r(3r-1)/2}.
\]
Now substitute $j=-r$. One has $r(3r-1)/2=j(3j+1)/2$ and $(-1)^{r+k}=(-1)^{j+k}$,
while the condition $r\notin[1-k,k]$ becomes $j\notin[-k,k-1]$. Hence the sum
equals the truncated tail
\[
  T_k(q):=\sum_{j\notin[-k,k-1]}(-1)^{j+k}q^{j(3j+1)/2},
\]
and therefore
\begin{equation}\label{eq:Ck-tail}
  C_k(q)=\frac{\prod_{m}(1-q^{2N_m})}{\poch}\,T_k(q).
\end{equation}

We rewrite \eqref{eq:Ck-tail} cleverly, and recalling the
definitions of $H$ and $B_k$ gives us our desired factorization
\begin{equation}\label{eq:factorization}
  C_k(q)
  =\underbrace{D(q)\,\frac{\prod_{m}(1-q^{2N_m})}{\poch}}_{H(q)}\cdot
   \underbrace{\frac{T_k(q)}{D(q)}}_{B_k(q)}
  =H(q)\,B_k(q).
\end{equation}
By Lemma~\ref{lem:liu} and Lemma~\ref{lem:H} both factors lie in $\NN[[q]]$, so
\eqref{eq:factorization} immediately gives $C_k(q)\in\NN[[q]]$. This proves the
nonnegativity assertion of Theorem~\ref{thm:main}.

We now show that $c_k(n)>0$ if and only if $n\ge g_k$. 

Recall that $B_k=T_k/D$, where the tail $T_k(q)$ runs over the two ranges
$j\ge k$ and $j\le -k-1$. On the first range the exponent $j(3j+1)/2$ is
strictly increasing, taking the values $g_k$ and
$(k+1)(3k+4)/2=g_k+(3k+2)$ at $j=k$ and $j=k+1$ respectively. On the second range,
substituting $j=-i$ with $i\ge k+1$ turns the exponent into the increasing
function $i(3i-1)/2$, whose two smallest values are
$(k+1)(3k+2)/2=g_k+(2k+1)$ and $(k+2)(3k+5)/2=g_k+(5k+5)$, at $i=k+1$ and
$i=k+2$ respectively. Since $2k+1<3k+2<5k+5$, the three smallest exponents of $T_k$ are
$g_k$, $g_k+(2k+1)$ and $g_k+(3k+2)$, coming from $j=k,\,-k-1,\,k+1$ with
signs $(-1)^{j+k}=+1,\,-1,\,-1$, and every other exponent is at least
$g_k+(5k+5)$. Hence
\begin{equation}\label{eq:Tk-head}
  T_k(q)=q^{g_k}-q^{\,g_k+2k+1}-q^{\,g_k+3k+2}+(\text{terms of degree}\ge g_k+5k+5).
\end{equation}

Next, we write
\[
  \frac{1}{D(q)}=\frac{1}{(1-q)(1-q^4)(1-q^5)}=\sum_{t\ge 0}\alpha(t)\,q^t,
\]
so that $\alpha(t)$ is the number of partitions of $t$ into parts $1$, $4$
and $5$. Counting these directly gives
\begin{equation}\label{eq:alpha-1}
  \alpha(0)=\alpha(1)=\alpha(2)=\alpha(3)=1,
  \qquad\text{and}\qquad \alpha(4)=2,\ \alpha(5)=3,\ \alpha(6)=3.
\end{equation}

Write $B_k(q)=\sum_{n\ge 0}b_k(n)q^n$, and adopt the convention
$\alpha(t)=0$ for $t<0$. Since $B_k=T_k\cdot\frac{1}{D}$, a term
$\varepsilon_E\, q^{E}$ of $T_k$ contributes exactly
$\varepsilon_E\,\alpha(n-E)$ to $b_k(n)$, so by the Cauchy product, we have
\[
  b_k(n)=\sum_{E}\varepsilon_E\,\alpha(n-E),
\]
the sum being over the exponents $E$ occurring in $T_k$, with
$\varepsilon_E\in\{1,-1\}$ the corresponding sign. In particular, only
those exponents with $E\le n$ contribute.

Now fix $t$ with $0\le t<2k+1$ and take $n=g_k+t$. By
\eqref{eq:Tk-head}, the smallest exponent of $T_k$ is $g_k$, and every
other exponent is at least $g_k+2k+1>g_k+t=n$. Hence the term
$+q^{g_k}$ is the only one contributing to $b_k(n)$ in this case, and
\begin{equation}\label{eq:bk-low}
  b_k(n):=b_k(g_k+t)=\alpha(t)\qquad(0\le t<2k+1).
\end{equation}
Combined with \eqref{eq:alpha-1}, this yields
\begin{equation}\label{eq:bk-values}
  b_k(g_k)=b_k(g_k+1)=b_k(g_k+2)=1\ \ (k\ge 1),
  \qquad
  b_k(g_k+3)=1\ \ (k\ge 2),
\end{equation}
the condition $k\ge 2$ in the last value coming from the requirement
$3<2k+1$. For $k=1$ the identity \eqref{eq:bk-low} is unavailable at
$t=3$; indeed $b_1(g_1+3)=\alpha(3)-\alpha(0)=0$. We handle the single remaining case below.

We write $H(q)=\sum\limits_{n\geq 0} h(n)q^n$ as before. Comparing coefficients in \eqref{eq:factorization} gives us
\begin{equation}\label{eq:conv-1}
  c_k(n)=\sum_{\ell\ge 0}h(\ell)\,b_k(n-\ell).
\end{equation}
Since every summand above is nonnegative, $c_k(n)=0$ when every summand vanishes.
Note that $c _k(n)>0$ follows from a single index $\ell$ with $h(\ell)>0$
and $b_k(n-\ell)>0$, for the latter we draw on the values $h(0)=h(3)=1$ and
$h(r)>0$ for $r\ge 7$ from Lemma~\ref{lem:H}, together with
\eqref{eq:bk-values}.

By \eqref{eq:Tk-head}, $B_k$ has no terms of degree below $g_k$. Thus for
$n<g_k$ every summand of \eqref{eq:conv-1} vanishes and $c_k(n)=0$, while at
$n=g_k$ only $\ell=0$ survives, giving us
\[
  c_k(g_k)=h(0)\,b_k(g_k)=1>0.
\]

Now we fix $r\ge 1$ and we will show $c_k(g_k+r)>0$. For $r\ge 7$, the part $h(r)>0$
pairs with the leading coefficient $b_k(g_k)=1$. For $1\le r\le 6$, the parts
$0$ and $3$ of $H$ combine with \eqref{eq:bk-values}, and we have
\[
\begin{array}{llll}
  r=1: & \ell=0, & h(0)\,b_k(g_k+1)>0 & (k\ge 1)\\[2pt]
  r=2: & \ell=0, & h(0)\,b_k(g_k+2)>0 & (k\ge 1)\\[2pt]
  r=3: & \ell=3, & h(3)\,b_k(g_k)>0 & (k\ge 1)\\[2pt]
  r=4: & \ell=3, & h(3)\,b_k(g_k+1)>0 & (k\ge 1)\\[2pt]
  r=5: & \ell=3, & h(3)\,b_k(g_k+2)>0 & (k\ge 1)\\[2pt]
  r=6: & \ell=3, & h(3)\,b_k(g_k+3)>0 & (k\ge 2)\\[2pt]
  r\ge 7: & \ell=r, & h(r)\,b_k(g_k)>0 & (k\ge 1)
\end{array}
\]
The remaining case is $r=6$ with $k=1$. Here $g_1=2$ and \eqref{eq:Tk-head} gives us
\[T_1(q)=q^{2}-q^{5}-q^{7}+(\text{terms of degree}\ge 12),\] so no term beyond
the displayed three influences the coefficient of $q^8$ in $B_1=T_1/D$, and
\eqref{eq:alpha-1} gives us
\[
  b_1(8)=\alpha(6)-\alpha(3)-\alpha(1)=3-1-1=1>0,
\]
hence, $c_1(2+6)\ge h(0)\,b_1(8)=1>0$.

Every $r\ge 1$ is thus accounted for, so together with the vanishing
below $g_k$ we now conclude
\[
  c_k(n)>0\quad\Longleftrightarrow\quad n\ge g_k=\frac{k(3k+1)}{2},
\]
completing the proof of Theorem~\ref{thm:main}. \qed

\section{Proof of Theorem \ref{thm:main3}}\label{sec:main3}

We shall follow the same general proof strategy as the proof of Theorem \ref{thm:main}.

Define
\begin{equation}\label{eq:cubic-product}
  \mathcal{P}_{\mathrm{t}}(q)
  :=\prod_{r\geq 1}
     \bigl(q^{2^r r};q^{2^{r+1}r}\bigr)_\infty.
\end{equation}
Merca \cite[Theorem~12]{Merca2025b} proved that
\begin{equation}\label{eq:merca-cubic-gf}
  \sum_{n\geq 0}\ddot{a}(n)q^n
  =\frac{\mathcal{P}_{\mathrm{t}}(q)}{(q;q)_\infty},
\end{equation}
where $\ddot{a}(n)$ counts the partitions of $n$ in which the
multiplicity of each part $m$ is at most
\(
  1+\nu_2(m^2)=2\nu_2(m)+1.
\)
To express \eqref{eq:merca-cubic-gf} in a form suitable for telescoping, we define
\begin{equation}\label{eq:f-cubic}
  f(m):=m\nu_2(4m^2)
       =m\bigl(2\nu_2(m)+2\bigr)=2m\bigl(\nu_2(m)+1\bigr).
\end{equation}

If $m=2^a u$, where $u$ is odd, then
\(
  f(m)=2^{a+1}(a+1)u.
\)
Consequently, we have
\begin{align}
  \prod_{m\geq 1}(1-q^{f(m)})
  &=
  \prod_{a\geq 0}
  \prod_{\substack{u\geq 1\\u\text{ odd}}}
  \left(1-q^{2^{a+1}(a+1)u}\right) \notag\\
  &=
  \prod_{r\geq 1}
  \prod_{\ell\geq 0}
  \left(1-q^{2^r r(2\ell+1)}\right) \notag\\
  &=
  \prod_{r\geq 1}
  \bigl(q^{2^r r};q^{2^{r+1}r}\bigr)_\infty
  =\mathcal{P}_{\mathrm{t}}(q).
  \label{eq:cubic-numerator}
\end{align}
It follows that
\begin{equation}\label{eq:cubic-quotient}
  \frac{\mathcal{P}_{\mathrm{t}}(q)}{(q;q)_\infty}
  =
  \prod_{m\geq 1}\frac{1-q^{f(m)}}{1-q^m}.
\end{equation}
Moreover,
\begin{equation}\label{eq:cubic-geometric}
  \frac{1-q^{f(m)}}{1-q^m}
  =
  1+q^m+q^{2m}+\cdots+q^{(2\nu_2(m)+1)m}.
\end{equation}

Fix $f^{(0)}(s)=s$ and define $f^{(r+1)}(s)=f(f^{(r)}(s))$. For $s\geq 1$, let
\[
  \mathcal{O}^{\mathrm{t}}_s
  :=\{f^{(r)}(s):r\geq 0\}
\]
be the forward orbit of $s$ under $f$, and set
\[
  \mathcal{R}_{\mathrm{t}}
  :=
  \mathcal{O}^{\mathrm{t}}_1
  \cup
  \mathcal{O}^{\mathrm{t}}_4
  \cup
  \mathcal{O}^{\mathrm{t}}_5.
\]
We now have the following result analogous to Lemma \ref{lem:H}.
\begin{lemma}\label{lem:H-cubic}
Let
\begin{equation}\label{eq:H-cubic-def}
  H_{\mathrm{t}}(q)
  :=
  (1-q)(1-q^4)(1-q^5)
  \frac{\mathcal{P}_{\mathrm{t}}(q)}{(q;q)_\infty}
  =
  \sum_{n\geq 0}h_{\mathrm{t}}(n)q^n.
\end{equation}
Then
\begin{equation}\label{eq:H-cubic-product}
  H_{\mathrm{t}}(q)
  =
  \prod_{m\notin\mathcal{R}_{\mathrm{t}}}
  \left(
    1+q^m+q^{2m}+\cdots+q^{(2\nu_2(m)+1)m}
  \right)
  \in\mathbb{N}[[q]].
\end{equation}

Moreover,
\begin{equation}\label{eq:H-cubic-support}
  h_{\mathrm{t}}(n)>0
  \quad\Longleftrightarrow\quad
  n\in\{0,3,6,7\}\ \text{or}\ n\geq 9.
\end{equation}
In particular,
\[
  h_{\mathrm{t}}(0)
  =h_{\mathrm{t}}(3)
  =h_{\mathrm{t}}(6)
  =h_{\mathrm{t}}(7)
  =1.
\]
\end{lemma}

\begin{proof}
The proof proceeds in four steps like the proof of Lemma \ref{lem:H}. We first show that the three orbits
$\mathcal{O}^{\mathrm{t}}_1,\mathcal{O}^{\mathrm{t}}_4,\mathcal{O}^{\mathrm{t}}_5$
are pairwise disjoint, so that regrouping the factors of
\eqref{eq:cubic-quotient} along these orbits is allowed. We then show that
the factors indexed by a single orbit $\mathcal{O}^{\mathrm{t}}_s$ telescope
down to $\tfrac1{1-q^s}$, then we assemble these pieces into the product
formula \eqref{eq:H-cubic-product}, and finally read off the support of
$H_{\mathrm{t}}$ from \eqref{eq:H-cubic-product} by interpreting
$h_{\mathrm{t}}(n)$ as a restricted partition counting function.

We will now show that the three orbits are pairwise disjoint. We first
describe each orbit explicitly. Write
\[
  x_r:=f^{(r)}(1),\qquad a_r:=\nu_2(x_r),
\]
so that the first few terms of the three orbits are
\[
  \mathcal{O}^{\mathrm{t}}_1=\{1,2,8,64,896,\dots\},\quad
  \mathcal{O}^{\mathrm{t}}_4=\{4,24,192,2688,\dots\},\quad
  \mathcal{O}^{\mathrm{t}}_5=\{5,10,40,320,4480,\dots\}.
\]
Since $f(m)=2m(\nu_2(m)+1)$ by \eqref{eq:f-cubic}, we have
$x_{r+1}=f(x_r)=2x_r(a_r+1)$. The integer $a_r+1$ is positive, so taking
$2$-adic valuations and using $\nu_2(mn)=\nu_2(m)+\nu_2(n)$ gives the
recurrence
\begin{equation}\label{eq:cubic-valuation-recurrence}
  a_{r+1}=1+a_r+\nu_2(a_r+1)>a_r,
\end{equation}
the inequality because $\nu_2(a_r+1)\ge 0$. Thus the sequence
$(a_r)_{r\ge 0}$ is strictly increasing. Its first three
values are $a_0=0$, $a_1=1$ and $a_2=3$, and therefore
\begin{equation}\label{eq:cubic-no-valuation-two}
  a_r\ne 2\qquad(r\ge 0).
\end{equation}

Next consider the effect of odd multipliers. If $u$ is odd then
$\nu_2(um)=\nu_2(m)$, so multiplying by $u$ does not change the valuation and
\begin{equation}\label{eq:cubic-odd-scaling}
  f(um)=u\,f(m).
\end{equation}
Thus $f$ commutes with multiplication by an odd number, and once an orbit
passes through an odd multiple of some $x_s$, all later iterates are obtained
by applying $f$ to that multiple. Since $f(5)=10=5x_1$, induction using
\eqref{eq:cubic-odd-scaling} lets us conclude
\begin{align}
  f^{(r)}(5)&=5x_r\qquad(r\ge 0),\label{eq:cubic-O5-scaling}\\
  f^{(r)}(4)&=3x_{r+1}\qquad(r\ge 1),\label{eq:cubic-O4-scaling}
\end{align}
the second of these because $f(4)=24=3x_2$.

We now verify disjointness. An equality $x_r=5x_s$ among the elements of
$\mathcal{O}^{\mathrm{t}}_1$ and $\mathcal{O}^{\mathrm{t}}_5$ would force
$a_r=a_s$, hence $r=s$ by the strict increase in
\eqref{eq:cubic-valuation-recurrence}, and then $x_r=5x_r$, which is
impossible. So
$\mathcal{O}^{\mathrm{t}}_1\cap\mathcal{O}^{\mathrm{t}}_5=\varnothing$. For
$\mathcal{O}^{\mathrm{t}}_4$ we treat the initial element separately: $4$ has
$2$-adic valuation exactly $2$, whereas by \eqref{eq:cubic-O5-scaling} every
element of $\mathcal{O}^{\mathrm{t}}_1\cup\mathcal{O}^{\mathrm{t}}_5$ has
valuation $a_r$, and $a_r\ne 2$ by \eqref{eq:cubic-no-valuation-two}. For the
remaining elements, \eqref{eq:cubic-O4-scaling} shows that an equality
$x_s=3x_{r+1}$ or $5x_s=3x_{r+1}$ would again force $s=r+1$, giving
$x_s=3x_s$ or $5x_s=3x_s$, both of which are impossible.

Thus, we have proved that
$\mathcal{O}^{\mathrm{t}}_1,\mathcal{O}^{\mathrm{t}}_4$ and
$\mathcal{O}^{\mathrm{t}}_5$ are pairwise disjoint.

The second part of the proof is exactly similar to the one in Lemma \ref{lem:H}, so we omit it for the sake of brevity. From the above disjointness of $\mathcal{O}^{\mathrm{t}}_1,\mathcal{O}^{\mathrm{t}}_4$ and
$\mathcal{O}^{\mathrm{t}}_5$, we can now conclude (like we did in the proof of Lemma \ref{lem:H}) that
\[
  \prod_{m\ge 1}\frac{1-q^{f(m)}}{1-q^m}
  =\Biggl(\prod_{s\in\{1,4,5\}}\frac{1}{1-q^{s}}\Biggr)
   \prod_{m\notin\mathcal{R}_{\mathrm{t}}}\frac{1-q^{f(m)}}{1-q^m}
  =\frac{1}{D(q)}
   \prod_{m\notin\mathcal{R}_{\mathrm{t}}}\frac{1-q^{f(m)}}{1-q^m}.
\]
Multiplying both sides by $D(q)$ cancels the prefactor, and rewriting each
remaining factor via \eqref{eq:cubic-geometric} in its geometric-series form
$1+q^m+\cdots+q^{(2\nu_2(m)+1)m}$ yields \eqref{eq:H-cubic-product}. Every
factor in \eqref{eq:H-cubic-product} is a polynomial with coefficients in
$\{0,1\}$, so the coefficients of $H_{\mathrm{t}}$ are nonnegative. This
proves the first part of the result.

We now want to prove the second part of the result. Expanding the product
\eqref{eq:H-cubic-product} shows that $h_{\mathrm{t}}(n)$ counts the
representations
\[
  n=\sum_{m\notin\mathcal{R}_{\mathrm{t}}}t_m\,m,
  \qquad 0\le t_m\le 2\nu_2(m)+1.
\]
In the language of partitions, $h_{\mathrm{t}}(n)$ is the number of
partitions of $n$ into \emph{allowed parts} $m\notin\mathcal{R}_{\mathrm{t}}$,
where the part $m$ may be repeated at most $2\nu_2(m)+1$ times. In
particular, $h_{\mathrm{t}}(n)>0$ if and only if at least one such partition
of $n$ exists. Note that every odd allowed part may be used at most once.

We first determine the small values. From the initial orbit values displayed
above we read off $\mathcal{R}_{\mathrm{t}}\cap[1,8]=\{1,2,4,5,8\}$, so the
allowed parts not exceeding $8$ are $3$, $6$ and $7$, the two odd ones being
usable at most once each. Thus, we have
\begin{itemize}
  \item $h_{\mathrm{t}}(0)=1$ (the empty partition), and
        $h_{\mathrm{t}}(3)=h_{\mathrm{t}}(6)=h_{\mathrm{t}}(7)=1$ (the single
        parts $3$, $6$ and $7$),
  \item $h_{\mathrm{t}}(1)=h_{\mathrm{t}}(2)=0$, since the smallest allowed
        part is $3$,
  \item $h_{\mathrm{t}}(4)=h_{\mathrm{t}}(5)=0$, since the only partitions
        available use the part $3$ at most once, and
  \item $h_{\mathrm{t}}(8)=0$, since $8\in\mathcal{R}_{\mathrm{t}}$ is not
        itself allowed, and none of $3+5$, $6+2$ and $7+1$ has both parts
        allowed.
\end{itemize}

We now want to show that $h_{\mathrm{t}}(n)>0$ for all $n\ge 9$. If
$n\notin\mathcal{R}_{\mathrm{t}}$, then $n$ is itself an allowed part, and
the one-part partition $n=n$ shows $h_{\mathrm{t}}(n)\ge 1$ for all such $n$.
Now let $n\in\mathcal{R}_{\mathrm{t}}$ with $n\ge 9$. Since $f(m)$ is always
even, the only odd elements of $\mathcal{R}_{\mathrm{t}}$ are the seeds $1$
and $5$, so $n$ is even. Then $n-3$ is odd with $n-3\ge 7$, hence
$n-3\notin\{1,5\}$, which forces $n-3\notin\mathcal{R}_{\mathrm{t}}$. Thus
$n=3+(n-3)$ is a partition of $n$ into two distinct allowed parts, and we
have $h_{\mathrm{t}}(n)\ge 1$ for every $n\ge 9$. This completes our proof.
\end{proof}

We are now ready to prove Theorem \ref{thm:main3}. We do this in two parts:
first, we show that $\tilde c_k(n)\geq 0$ for all $n\geq 0$, and then we look
at when we have strict inequality $\tilde c_k(n)>0$.

Subtracting the partial sum over $r\in[1-k,k]$ from \eqref{eq:epnt} gives
\[
  \poch-\sum_{r=1-k}^{k}(-1)^r q^{r(3r-1)/2}
  =\sum_{r\notin[1-k,k]}(-1)^r q^{r(3r-1)/2},
\]
so that
\[
  1-\frac{1}{\poch}\sum_{r=1-k}^{k}(-1)^r q^{r(3r-1)/2}
  =\frac{1}{\poch}\sum_{r\notin[1-k,k]}(-1)^r q^{r(3r-1)/2}.
\]
Multiplying by $(-1)^k\mathcal{P}_{\mathrm{t}}(q)$ and absorbing the sign into
the summand yields
\[
  \tilde C_k(q)
  =\frac{\mathcal{P}_{\mathrm{t}}(q)}{\poch}
   \sum_{r\notin[1-k,k]}(-1)^{r+k}q^{r(3r-1)/2}.
\]
Now substitute $j=-r$. One has $r(3r-1)/2=j(3j+1)/2$ and
$(-1)^{r+k}=(-1)^{j+k}$, while the condition $r\notin[1-k,k]$ becomes
$j\notin[-k,k-1]$. Hence the sum is exactly the tail
$T_k(q)=\sum_{j\notin[-k,k-1]}(-1)^{j+k}q^{j(3j+1)/2}$ of
Section~\ref{sec:reduction}, and therefore
\begin{equation}\label{eq:C-cubic-tail}
  \tilde C_k(q)=\frac{\mathcal{P}_{\mathrm{t}}(q)}{\poch}\,T_k(q).
\end{equation}

We rewrite \eqref{eq:C-cubic-tail} cleverly, and recalling the definitions of
$H_{\mathrm{t}}$ in \eqref{eq:H-cubic-def} and of $B_k$ gives us our desired
factorization
\begin{equation}\label{eq:C-cubic-factorization}
  \tilde C_k(q)
  =\underbrace{D(q)\,\frac{\mathcal{P}_{\mathrm{t}}(q)}{\poch}}_{H_{\mathrm{t}}(q)}
   \cdot\underbrace{\frac{T_k(q)}{D(q)}}_{B_k(q)}
  =H_{\mathrm{t}}(q)\,B_k(q).
\end{equation}
By Lemma~\ref{lem:liu} and Lemma~\ref{lem:H-cubic} both factors lie in
$\NN[[q]]$, so \eqref{eq:C-cubic-factorization} immediately gives
$\tilde C_k(q)\in\NN[[q]]$. This proves the nonnegativity assertion of
Theorem~\ref{thm:main3}.

We now show that $\tilde c_k(n)>0$ if and only if $n\ge g_k$.

The tail $T_k$ and the quotient $B_k=T_k/D$ are the same as in
Section~\ref{sec:reduction}, so we reuse the information already obtained
there. By \eqref{eq:Tk-head},
\[
  T_k(q)=q^{g_k}+(\text{terms of degree}\ge g_k+2k+1),
\]
and, writing $B_k(q)=\sum_{n\ge 0}b_k(n)q^n$ as before, \eqref{eq:bk-low}
together with $\alpha(0)=\alpha(1)=\alpha(2)=1$ gives
\begin{equation}\label{eq:cubic-b-low}
  b_k(g_k)=b_k(g_k+1)=b_k(g_k+2)=1\qquad(k\ge 1),
\end{equation}
the range being available for every $k\ge 1$ since $2<2k+1$. Moreover $B_k$
has no terms of degree below $g_k$.

Comparing coefficients in \eqref{eq:C-cubic-factorization} gives us
\begin{equation}\label{eq:cubic-convolution}
  \tilde c_k(n)=\sum_{\ell\ge 0}h_{\mathrm{t}}(\ell)\,b_k(n-\ell).
\end{equation}
Since every summand above is nonnegative, $\tilde c_k(n)=0$ when every
summand vanishes. While $\tilde c_k(n)>0$ follows from a single index $\ell$
with $h_{\mathrm{t}}(\ell)>0$ and $b_k(n-\ell)>0$, for the latter we draw on
the values $h_{\mathrm{t}}(0)=h_{\mathrm{t}}(3)=h_{\mathrm{t}}(6)
=h_{\mathrm{t}}(7)=1$ and $h_{\mathrm{t}}(r)>0$ for $r\ge 9$ from
Lemma~\ref{lem:H-cubic}, together with \eqref{eq:cubic-b-low}.

Thus for $n<g_k$ every summand of \eqref{eq:cubic-convolution} vanishes and
$\tilde c_k(n)=0$, while at $n=g_k$ only $\ell=0$ survives, giving us
\[
  \tilde c_k(g_k)=h_{\mathrm{t}}(0)\,b_k(g_k)=1>0.
\]

Now we fix $r\ge 1$ and we will show $\tilde c_k(g_k+r)>0$. For $r\ge 9$, the
part $h_{\mathrm{t}}(r)>0$ pairs with the leading coefficient $b_k(g_k)=1$.
For $1\le r\le 8$, the parts $0$, $3$, $6$ and $7$ of $H_{\mathrm{t}}$
combine with \eqref{eq:cubic-b-low}, and we have
\[
\begin{array}{llll}
  r=1: & \ell=0, & h_{\mathrm{t}}(0)\,b_k(g_k+1)>0 & (k\ge 1)\\[2pt]
  r=2: & \ell=0, & h_{\mathrm{t}}(0)\,b_k(g_k+2)>0 & (k\ge 1)\\[2pt]
  r=3: & \ell=3, & h_{\mathrm{t}}(3)\,b_k(g_k)>0 & (k\ge 1)\\[2pt]
  r=4: & \ell=3, & h_{\mathrm{t}}(3)\,b_k(g_k+1)>0 & (k\ge 1)\\[2pt]
  r=5: & \ell=3, & h_{\mathrm{t}}(3)\,b_k(g_k+2)>0 & (k\ge 1)\\[2pt]
  r=6: & \ell=6, & h_{\mathrm{t}}(6)\,b_k(g_k)>0 & (k\ge 1)\\[2pt]
  r=7: & \ell=7, & h_{\mathrm{t}}(7)\,b_k(g_k)>0 & (k\ge 1)\\[2pt]
  r=8: & \ell=6, & h_{\mathrm{t}}(6)\,b_k(g_k+2)>0 & (k\ge 1)\\[2pt]
  r\ge 9: & \ell=r, & h_{\mathrm{t}}(r)\,b_k(g_k)>0 & (k\ge 1)
\end{array}
\]

Every $r\ge 1$ is thus accounted for, so together with the vanishing below
$g_k$ we now conclude
\[
  \tilde c_k(n)>0\quad\Longleftrightarrow\quad n\ge g_k=\frac{k(3k+1)}{2},
\]
completing the proof of Theorem~\ref{thm:main3}. \qed

\section{Proofs of Theorems \ref{thm:main2} and \ref{thm:main4}}\label{sec:new}

We shall use Theorems \ref{thm:main} and \ref{thm:main3} to prove Theorem \ref{thm:main2} and \ref{thm:main4} respectively.

\begin{proof}[Proof of Theorem \ref{thm:main2}]
    Notice that
    \[
    \overline{C}_k(q)-C_{k+1}(q)=(-1)^{k-1}\dfrac{(-1)^{k+1}q^{(k+1)(3k+2)/2}}{(q;q)_\infty}\prod_{m\geq 1}(1-q^{2N_m}).
    \]
    This immediately implies
    \[
    \overline{C}_k(q)=C_{k+1}(q)+q^{g_k+2k+1}\sum_{n\geq 0}p_{\nu}(n)q^n.
    \]
    From Theorem \ref{thm:main} and the above equation (with the fact that $p_{\nu}(n)> 0$ for all $n\geq 0$), we immediately get $\overline{C}_k(q)\in \NN[[q]]$.

    For the exact positivity bound, we notice that the coefficients of the second summand is strictly positive for every $n\geq g_k+2k+1$. Also, $p_{\nu}(0)= 1$, which ensures that $q^{g_k+2k+1}$ is present in the series. Again, $C_{k+1}(q)$ starts at $g_{k+1}=g_k+3k+2$, which is greater than $g_k+2k+1$, giving us the positivity of the other part when $n\geq g_k+2k+1$. This proves the result.
\end{proof}

\begin{proof}[Proof of Theorem \ref{thm:main4}]
    Notice that
    \[
    \hat C_k(q)-\tilde C_{k+1}(q)=(-1)^{k-1}\dfrac{(-1)^{k+1}q^{(k+1)(3k+2)/2}}{(q;q)_\infty}\prod _{n=1}^\infty\bigl(q^{2^nn};q^{2^{n+1}n}\bigr)_\infty.
    \]
    This immediately implies
    \[
    \hat{C}_k(q)=\tilde C_{k+1}(q)+q^{g_k+2k+1}\sum_{n\geq 0}\ddot{a}(n)q^n.
    \]
    From Theorem \ref{thm:main3} and the above equation (with the fact that $\ddot{a}(n)>0$ for all $n\geq 0$), we immediately get $\hat {C}_k(q)\in \NN[[q]]$.

    For the exact positivity bound, we notice that the coefficients of the second summand is strictly positive for every $n\geq g_k+2k+1$. Also, $\ddot{a}(0)= 1$, which ensures that $q^{g_k+2k+1}$ is present in the series. Again, $\tilde C_{k+1}(q)$ starts at $g_{k+1}$ which is greater than $g_k+2k+1$, giving us the positivity of the other part when $n\geq g_k+2k+1$. This proves the result.
\end{proof}

\section*{Acknowledgments}

The authors used OpenAI's ChatGPT 5.5 and 5.6 Sol to clean up notations and editing of the manuscript. ChatGPT 5.6 Sol also caught the error in the original statement of Merca's Conjecture 15 \cite{Merca2025b} and gave the idea for the proof of Theorem \ref{thm:main2} presented here. The authors then verified the proof. The authors take full responsibility for all of the mathematical content of the paper. The first author thanks Ahmedabad University for funding the ChatGPT Plus license.

\end{document}